\documentclass{amsart}
\usepackage{amscd,amsmath,amssymb}
\usepackage{pstricks}
\usepackage{latexsym,amsbsy,mathrsfs}
\usepackage{xy}
\usepackage{xypic}
\usepackage{pgf,tikz}
\newtheorem{thm}{Theorem}[section]
\newtheorem{lem}[thm]{Lemma}
\newtheorem{cor}[thm]{Corollary}

\theoremstyle{definition}

\newtheorem{defn}[thm]{Definition}

\newtheorem{rem}[thm]{Remark}

\numberwithin{equation}{thm}
\begin{document}
\title[Recollement of Additive Quotient Categories ]
{Recollement of Additive Quotient Categories}

\author{Minxiong Wang and Zengqiang Lin}
\address{ School of Mathematical sciences, Huaqiao University,
Quanzhou\quad 362021,  China.} \email{mxw@hqu.edu.cn}
\address{ School of Mathematical sciences, Huaqiao University,
Quanzhou\quad 362021,  China.} \email{lzq134@163.com}

\thanks{This work was supported  by  the Natural Science Foundation of Fujian Province (Grants No. 2013J05009) and the Science Foundation of Huaqiao University (Grants No. 2014KJTD14)}



\begin{abstract} In this note,  we define  a recollement of additive categories, and prove that such a recollement
can induce a recollement of their quotient categories. As an application, we get a  recollement of quotient triangulated  categories induced by mutation pairs.
\end{abstract}

\maketitle

\section{Introduction}

 The recollement of triangulated categories was  introduced in a geometric setting by
 Beilinson, Bernstein, and  Deligne \cite{[BBD]}, and has been studied in an algebraic setting
 by  Cline,   Parshall and   Scott \cite{[CPS1],[CPS2]}.
The  recollement of abelian category was formulated by  Franjou and Pirashvili \cite{[FP]}.
The recollment of categories plays an important role in algebraic geometry and in representation theory.

Quotient categories give a way to produce abelian categories \cite{[KZ],[Na]}. Such as, Koenig and Zhu showed that the quotient category $\mathcal{C}/\mathcal{T}$ is abelian provided that $\mathcal{T}$ is a cluster tilting subcategory of a triangulated category $\mathcal{C}$ \cite{[KZ]}. A recollement of triangulated categories can induced a recollement of quotient abelian categories \cite{[Ch],[LW]}. On the other hand, quotient categories also give a way to produce triangulated categories \cite{[Ha],[IY]}. Such as, Iyama and Yoshino proved that if $(\mathcal{Z},\mathcal{Z})$ is a $\mathcal{D}$-mutation pair and $\mathcal{Z}$ is extension-closed, then $\mathcal{Z}/\mathcal{D}$ is a triangulated category \cite{[IY]}. It is natural to ask that whether a recollement of categories can induce a recollement of quotient triangulated categories.

In this note, to unify the recollment of abelian categories and the recollement of triangulated categories, we define the notion of recollement of addtive categories. In section 2, we  prove that a recollement of additive categories can induce a recollement of their subcategories and a recollement of their quotient categories under certain conditions.
 In section 3, we apply the main result of section 2 to triangulated  categories and obtain a recollement of  quotient triangulated  categories induced by mutation pairs.

\section{Recollement of additive categories}
Throughout this note,  all  subcategories of a category are full and closed under isomorphisms,
direct sums and direct summands. Let $F:\mathcal{A}\rightarrow \mathcal{B}$ be a functor. We denote by
Im$F$ the subcategory of $\mathcal{B}$ generated by objects $F(X)$ with $X\in\mathcal{A}$, and by
Ker$F$  the subcategory of $\mathcal{A}$ generated by objects $X$ with $F(X)=0$. Let $f:X\rightarrow Y$, $g:Y\rightarrow Z$
be two morphisms in $\mathcal{A}$, and $M$ an object in $\mathcal{C}$, we denote by $gf$ the composition of
$f$ and $g$,  by $f_{\ast}$ the morphism $\mbox{Hom}_{\mathcal{A}}(M,f):
\mbox{Hom}_\mathcal{A}(M,X)\rightarrow \mbox{Hom}_{\mathcal{A}}(M,Y)$, and by $f^{\ast}$ the morphism $\mbox{Hom}_{\mathcal{A}}(f,M):
\mbox{Hom}_\mathcal{A}(Y,M)\rightarrow \mbox{Hom}_{\mathcal{A}}(X,M)$.

\begin{defn}\label{defn2.1}
 Let $\mathcal {A}, \mathcal {A}'$ and $\mathcal {A}''$ be
additive categories. The diagram
 $$  \quad \quad \quad \quad \quad\quad \quad \quad \quad \quad \xymatrix @C=2.7pc {
      \mathcal {A}' \ar[r]^{i_*} & \mathcal {A}   \ar@<+2.8ex>[l]_{i^!}\ar@<-2.8ex>[l]_{i^*}\ar[r]^{j^*}
      & \mathcal {A}'' \ar@<+2.8ex>[l]_{j_*}\ar@<-2.8ex>[l]_{j_!}
      }
       \quad \quad \quad \quad \quad\quad \quad \quad \quad \quad  (2.1)$$
of additive functors is a {\em recollement of additive category $\mathcal{A}$ relative
to  additive categories $\mathcal{A}'$ and $\mathcal {A}''$},  if the following conditions are satisfied:
\par
(R1) $(i^{*}, i_{*}), (i_{*}, i^{!}), (j_{!}, j^{*})$ and $(j^{*}, j_{*})$ are adjoint pairs;
\par
(R2) $i_{*}, j_{!}$ and $j_{*}$ are full embeddings;
\par
(R3) Im$i_{*}$=Ker$j^{*}$.
\end{defn}
\par
\begin{rem} \label{rem}
(1) If all involved categories are abelian, then the diagram (2.1) is a recollement of abelian categories. If all
involved categories are triangulated,  and all involved functors are exact,  then by \cite[Theorem 3.2]{[Z]}, the diagram (2.1) is a recollement of triangulated categories in the sense of  Beilinson-Bernstein-Deligne.
\par
(2) For any adjoint pair $(F,G)$, it is well-known that $F$ is fully faithful if and only if the unit
$\varepsilon: id\rightarrow GF$ is a natural equivalence.
If  $F$ is a full embedding, then there exists an
 adjoint pair $(F,G')$ such that the unit $\varepsilon: id\rightarrow GF$ is the identity.
Thus if necessary, we may assume that $i^*i_*=id, i^!i_*=id$ and $j^*j_*=id$ in Definition \ref{defn2.1}.
\end{rem}

Let $\mathcal {A}$ be an additive category and $\mathcal {X}$ a subcategory. Then by definition the quotient category
${\mathcal {A}}/{\mathcal {X}}$ has the same objects as $\mathcal{A}$, and the set of morphisms from $A$ to $B$ in the quotient category is defined as the quotient group $\mbox{Hom}_{\mathcal{A}}(A,B)/[\mathcal{X}](A,B)$, where $[\mathcal{X}](A,B)$ is the subgroup of morphism in $\mathcal{A}$ factoring through some object in $\mathcal{X}$. For $f:A\rightarrow B$ a morphism in $\mathcal{A}$,
we denote by $\overline{f}$ its residue class in the quotient category.

The following lemma is well-known.
\par
\begin{lem} \label{lem3}
Let $\mathcal {A}$ be an additive category and $\mathcal {X}$ a subcategory.
There exists  an additive functor $Q_{\mathcal{A}}:\mathcal {A}\rightarrow {\mathcal {A}}/{\mathcal {X}}$ such that
\par
(1)  $Q_{\mathcal{A}}(X)$ is null for every $X\in\mathcal{X}$;
\par
(2)  For any additive category $\mathcal {B}$ and additive functor $F:\mathcal {A}\rightarrow \mathcal{B}$ such
that $F(X)$ is null for every $X\in\mathcal{X}$,  there exists a unique additive functor
$\widetilde{F}:{\mathcal {A}}/{\mathcal {X}}\rightarrow \mathcal {B}$ such that $F=\widetilde{F}\cdot Q_{\mathcal{A}}$,  i.e.,   we have the following commutative diagram of functors.
$$
\xymatrix{
  \mathcal{A} \ar[d]_{Q_{\mathcal{A}}} \ar[r]^{F} & \mathcal {B}       \\
  {\mathcal {A}}/{\mathcal {X}} \ar@{.>}[ur]_{\widetilde{F}}                     }
$$
\end{lem}
The following lemma is basic, but it is important in our proof. For convenience, we give a proof.
\begin{lem}  \label{lem4}
Let $\mathcal {A}$ and $\mathcal {A}'$ be
additive categories,  $i^*:\mathcal{A}\rightarrow
\mathcal{A}'$ and $i_*:\mathcal{A}'\rightarrow \mathcal{A}$ be additive
functors. If $\mathcal {X}$  is a subcategory of $\mathcal{A}$ and  $\mathcal {X}'$ a subcategories of
 $\mathcal {A}'$, satisfying $i^*(\mathcal {X})\subset \mathcal {X}'$ and $i_*(\mathcal {X}')\subset \mathcal {X}$,
then
\par
(1) The functor $i^*$ induces an additive functor  $\widetilde{i^*}:\mathcal{A}/{\mathcal {X}} \rightarrow \mathcal{A}'/\mathcal {X}'$;
 \par
 (2) The functor $i_*$ induces  an additive functor $\widetilde{i_*}:\mathcal{A}'/\mathcal {X}' \rightarrow \mathcal{A}/{\mathcal {X}}$;
  \par
 (3)  If $(i^*, i_*)$ is an adjoint pair,  then so is $(\widetilde{i^*}, \widetilde{i_*})$.
\end{lem}
\begin{proof}
(1)Since $Q_{\mathcal{A}'}i^*(X)$ is null for any $X\in \mathcal{X}$, by Lemma \ref{lem3}
there exists a unique additive functor
 $\widetilde{i^*}:\mathcal{A}/{\mathcal {X}} \rightarrow \mathcal{A}'/\mathcal {X}'$ such that the following
 diagram is commutative
 $$\xymatrix {
 \mathcal{A} \ar[d]_{Q_{\mathcal{A}}} \ar[rr]^{i^*} && \mathcal{A'} \ar[d]^{Q_{\mathcal{A}'}} \\
  \mathcal{A}/{\mathcal {X}} \ar@{.>}[rr]^{\widetilde{i^*}} && \mathcal{A}'/\mathcal {X}'.}
$$

\par
We can prove (2) similarly.
\par
(3) Note that $(i^*, i_*)$ is an adjoint pair ,  there exist two natural
isomorphisms $$
\eta_{A, A'}:\mbox{Hom}_{\mathcal{A'}}(i^*(A), A')\rightarrow
\mbox{Hom}_{\mathcal{A}}(A, i_*(A')), $$ $$
\tau_{A, A'}:\mbox{Hom}_{\mathcal{A}}(A, i_*(A'))\rightarrow
\mbox{Hom}_{\mathcal{A'}}(i^*(A), A') $$
such that $\tau_{A,A'}\eta_{A,A'}=1$ and $\eta_{A,A'}\tau_{A,A'}=1$, where $A\in
\mathcal{A},  A'\in\mathcal{A'}$.
We claim that $\eta_{A, A'}$ induces a natural isomorphism $$
\widetilde{\eta}_{A, A'}:\mbox{Hom}_{\mathcal{A'}/\mathcal {X}'}(\widetilde{i^*}(A), A')\rightarrow
\mbox{Hom}_{\mathcal{A}/{\mathcal{X}}}(A, \widetilde{i_*}(A')).$$
\par  We first define a map $
\widetilde{\eta}_{A, A'}:\mbox{Hom}_{\mathcal{A'}/\mathcal {X}'}(\widetilde{i^*}(A), A')\rightarrow
\mbox{Hom}_{\mathcal{A}/{\mathcal{X}}}(A, \widetilde{i_*}(A')) $. For
any $\overline{f}\in
\mbox{Hom}_{\mathcal{A'}/\mathcal {X}'}(\widetilde{i^*}(A), A')$,
define
$\widetilde{\eta}_{A, A'}(\bar{f}):=\overline{\eta_{A, A'}(f)}$. In
order to show that the definition is reasonable,  it suffices to
prove that $\overline{\eta_{A, A'}(f)}$ does not depend on the choice
of $f$. Let $\overline{f_1}=\overline{f_2}\in
\mbox{Hom}_{\mathcal{A'}/\mathcal {X}'}(\widetilde{i^*}(A), A')$,
i.e.,  there exist two morphisms $s: i^*(A)\rightarrow X$ and $t: X'\rightarrow A'$ such that the following diagram is
commutative $$ \xymatrix{
  {i^*}(A) \ar[rr]^{f_1-f_2} \ar[dr]_{s}    &  &    A'    \\
                & X'\ar[ur]_{t}                  }
$$
where $X'\in \mathcal{X}'$.
It is easy to see that ${\eta}_{A, A'}(f_1)-{\eta}_{A, A'}(f_2)={\eta}_{A, A'}(f_1-f_2)
=i_*(f_1-f_2)\varepsilon_A=i_*(ts)\varepsilon_A=i_*(t)i_*(s)\varepsilon_A$,
where $\varepsilon_A$ is the adjunction morphism.
Hence the following diagram is commutative
$$
\xymatrix{
  A \ar[rr]^{{\eta}_{A, A'}(f_1)-{\eta}_{A, A'}(f_2)} \ar[dr]_{i_*(s)\varepsilon_A}    &  &    i_*(A')    \\
                & i_*(X')\ar[ur]_{i_*(t)}                  }
$$
Since $i_*(\mathcal {X}')\subset \mathcal {X}$, we have  $i_*(X')\in \mathcal{X}$, thus
$\overline{\eta_{A, A'}(f_1)}=\overline{\eta_{A, A'}(f_2)}$. Therefore $\widetilde{\eta}_{A, A'}$ is a morphism
 between $\mbox{Hom}_{\mathcal{A'}/\mathcal {X}'}(\widetilde{i^*}(A), A')$ and
 $\mbox{Hom}_{\mathcal{A}/{\mathcal{X}}}(A, \widetilde{j_*}(A'))$.

 Similarly,  we can define a morphism
$$
\widetilde{\tau}_{A, A'}:\mbox{Hom}_{\mathcal{A}/{\mathcal{X}}}(A, \widetilde{i_*}(A'))\rightarrow
\mbox{Hom}_{\mathcal{A'}/\mathcal {X}'}(\widetilde{i^*}(A), A')
$$
by $\widetilde{\tau}_{A, A'}(g):=\overline{\tau_{A, A'}(g)}$, where $\overline{g}\in \mbox{Hom}_{\mathcal{A}/{\mathcal{X}}}(A, \widetilde{i_*}(A'))$.

Note that $\tau_{A,A'}\eta_{A,A'}=1$ and $\eta_{A,A'}\tau_{A,A'}=1$,
we obtain that $\tilde{\tau}_{A,A'} \tilde{\eta}_{A,A'}=1$ and $\tilde{\eta}_{A,A'} \tilde{\tau}_{A,A'}=1$.
Therefore $\tilde{\eta}_{A,A'}$ is a bijection.
\par
It remains to prove that $\widetilde{\eta}_{A, A'}$ is a natural transformation.
For any morphism  $\overline{h}\in \mbox{Hom}_{\mathcal{A}/{\mathcal{X}}}(B, A)$,
 we claim that the following diagram is commutative
 $$
 \xymatrix {
  \mbox{Hom}_{\mathcal{A'}/\mathcal {X}'}(\widetilde{i^*}(A), A')
  \ar[d]_{(\widetilde{i^*}\bar{h})^* } \ar[rr]^{\widetilde{\eta}_{A, A'}} &&
  \mbox{Hom}_{\mathcal{A}/{\mathcal{X}}}(A, \widetilde{i_*}(A'))  \ar[d]^{\bar{h}^*} \\
  \mbox{Hom}_{\mathcal{A'}/\mathcal {X}'}(\widetilde{i^*}(B), A') \ar[rr]^{\widetilde{\eta}_{B, A'}} &&
  \mbox{Hom}_{\mathcal{A}/{\mathcal{X}}}(B, \widetilde{i_*}(A'))
 }$$
  In fact, for any morphism
$\bar{f}\in\mbox{Hom}_{\mathcal{A'}/\mathcal {X}'}(\widetilde{i^*}(A), A')$,
since $\eta_{A, A'}$ is a natural transformation,
we have $\eta_{A,A'}(f)\cdot h=\eta_{B,A'}(f\cdot i^*h)$. Thus  $$\tilde{\eta}_{A,A'}(\bar{f})\cdot \bar{h}
=
\overline{\eta_{A,A'} (f)\cdot h}=\overline{\eta_{B,A'} (f\cdot i^*(h))}
=\widetilde{\eta}_{B,A'}\overline{ (f\cdot i^*(h))}
=\widetilde{\eta}_{B,A'}{ (\bar{f}\cdot \widetilde{i^*}(\overline{h}))},$$  that is,
$\widetilde{\eta}_{B,A'}\cdot(\widetilde{i^*}\bar{h})^*=\bar{h}^*\cdot\widetilde{\eta}_{A,A'}$. Hence $\widetilde{\eta}_{A, A'}$ is
natural in the first variable.
\par
Using similar arguments as before one can show that $\widetilde{\eta}_{A, A'}$ is  natural in the second variable.
This  finishes the proof.
\end{proof}

Now we can state and prove our main theorem in this section.

 \begin{thm} \label{thm5}
 Let
 $$ \xymatrix @C=2.7pc {
      \mathcal {A}' \ar[r]^{i_*} & \mathcal {A}   \ar@<+2.8ex>[l]_{i^!}\ar@<-2.8ex>[l]_{i^*}\ar[r]^{j^*}
      & \mathcal {A}'' \ar@<+2.8ex>[l]_{j_*}\ar@<-2.8ex>[l]_{j_!}
      }
$$
 be a recollement of additive categories,  $\mathcal {X}$  an additive subcategory of
$\mathcal {A}$ such that $i_*i^*(\mathcal {X})\subset \mathcal {X}$,
$j_*j^*(\mathcal {X})\subset \mathcal {X}$,
$i_*i^!(\mathcal {X})\subset \mathcal {X}$ and
$j_!j^*(\mathcal {X})\subset \mathcal {X}$. Then
\par
(1) the diagram $$ \xymatrix @C=2.7pc {
     i^* (\mathcal {X}) \ar[r]^-{\bar{i}_*} &
      \mathcal {X}   \ar@<+2.8ex>[l]_-{\bar{i}^!}\ar@<-2.8ex>[l]_-{\bar{i}^*}\ar[r]^-{\bar{j}^*}
      & j^*(\mathcal {X}) \ar@<+2.8ex>[l]_-{\bar{j}_*}\ar@<-2.8ex>[l]_-{\bar{j}_!}
      }
$$ is a recollement of additive categories,  where $\bar{i}_*, \bar{i}^!, \bar{i}^*, \bar{j}^*,
 \bar{j}_*$ and $\bar{j}_!$ are  restriction functors of  $i_*, i^!, i^*, j^*, j_*$ and $j_!$ respectively;

\par
(2) there exists a diagram of additive functors $$ \quad \quad \quad\quad \quad \quad \quad \xymatrix @C=2.7pc {
      \mathcal {A}'/i^*(\mathcal{X}) \ar[r]^-{\tilde{i}_*} & \mathcal {A}/\mathcal{X}
      \ar@<+2.8ex>[l]_-{\tilde{i}^!}\ar@<-2.8ex>[l]_-{\tilde{i}^*}\ar[r]^-{\tilde{j}^*}
      & \mathcal {A}''/ j^*(\mathcal{X}) \ar@<+2.8ex>[l]_-{\tilde{j}_*}\ar@<-2.8ex>[l]_-{\tilde{j}_!}}
    \quad \quad \quad \quad \quad\quad \quad  (2.2)
$$ Moreover, the diagram (2.2) is a recollement of additive categories if and only if $\mathcal{X}\subset$ Ker$j^*$.
\end{thm}
 \begin{proof}
(1) (R1) and (R2) are trivial. For (R3), it is clear that  $\bar{j}^*\bar{i}_*=0$,  hence Im$\bar{i}_*\subset$ Ker$\bar{j}^*$.
On the other hand, since Ker$\bar{j}^*$=Ker$j^*\bigcap \mathcal{X}$=Im $i_*\bigcap \mathcal{X}$,  for any $X\in $ Ker$\bar{j}^*$,
  there exists $A'\in \mathcal{A}' $ such that $X$=$i_*(A')$. It follows that $X=i_*(A')= i_* i^* i_*(A')=i_*i^*(X)\in$
 Im$\bar{i}_*$.
 This shows that Im$\bar{i}_*=$ Ker$\bar{j}^*$.

 \par
 (2) By Lemma \ref{lem4}, the six functors $i^*, i_*, i^!, j_!, j^*, j_*$ induce six additive
functors  $\tilde{i^*}, \tilde{i_*}, \tilde{i^!}, \tilde{j_!}, \tilde{j^*}, \tilde{j_*}$,  respectively.
Furthermore,  $(\tilde{i^*}, \tilde{i_*})$, $(\tilde{i_*}, \tilde{i^!})$,  $(\tilde{j_!}, \tilde{j^*})$  and  $(\tilde{j^*}, \tilde{j_*})$ are adjoint pairs.
Since $i_*$ is a full embedding, we may assume that  $i^!i_*=
id_{\mathcal{A}'}$. Thus $\tilde{i^!}\tilde{i_*}=
id_{\mathcal{A}'/i^*{\mathcal(X)}}$. It follows that $\tilde{i_*}$
is a full embedding. Similarly,  $\tilde{j_!}, \tilde{j_*}$ are full
embeddings. Since $j^\ast i_\ast=0$, we have $\widetilde{j}^\ast\widetilde{i}_\ast=0$, which implies that Im$\widetilde{i}_\ast\subset\mbox{Ker}\widetilde{j}^\ast$. To end the proof, it remains to prove that $\mbox{Ker}\widetilde{j}^\ast=\mbox{Im}\widetilde{i}_\ast$ if and only if $\mathcal{X} \subset$ Ker$j^*$.
\par
In fact, if  $\mathcal{X} \subset$ Ker$j^*$,  we have
Ker$\widetilde{j}^*=j^{*-1}(j^*(\mathcal{X}))/\mathcal{X}$=Ker${j}^*/\mathcal{X}$=Im${i}_*/\mathcal{X}$=Im$\widetilde{i}_*$.
On the other hand, assume that Im$\widetilde{i}_*$=Ker$\widetilde{j}^*$.
 By the definition of quotient category and that of quotient functor,
 Im$\widetilde{i}_*$ and Im$i_*$ have the same object,  and the objects of $\mathcal{X}$ is a subclass of the objects of Ker$\widetilde{j}^*$.
 Since Im${i}_*$=Ker${j}^*$,   the objects of $\mathcal{X}$ is a subclass
 of the objects of Ker${j}^*$. Therefore, $\mathcal{X}\subset$ Ker$j^*$.
\par

This finishes the proof.
\end{proof}

 \begin{cor} \label{cor6}
  Let
 $$ \xymatrix @C=2.7pc {
      \mathcal {A}' \ar[r]^{i_*} & \mathcal {A}   \ar@<+2.8ex>[l]_{i^!}\ar@<-2.8ex>[l]_{i^*}\ar[r]^{j^*}
      & \mathcal {A}'' \ar@<+2.8ex>[l]_{j_*}\ar@<-2.8ex>[l]_{j_!}
      }
$$
 be a recollement of additive categories. Let $\mathcal {X}'$ be a subcategory of $\mathcal {A}'$ and $\mathcal {X}''$ a subcategory of $\mathcal {A}''$, satisfying $i^*j_*(\mathcal {X}'')\subset \mathcal {X}'$ and $i^!j_!(\mathcal {X}'')\subset \mathcal {X}'$. If  $\mathcal {X}=\{X\in \mathcal {A}|j^*(X)\in \mathcal {X}'', i^*(X)\in \mathcal {X}', i^!(X)\in \mathcal {X}'\}$,
Then $$ \xymatrix @C=2.7pc {
     \mathcal {X}' \ar[r]^{\bar{i}_*} &
      \mathcal {X}   \ar@<+2.8ex>[l]_{\bar{i}^!}\ar@<-2.8ex>[l]_{\bar{i}^*}\ar[r]^{\bar{j}^*}
      & \mathcal {X}'' \ar@<+2.8ex>[l]_{\bar{j}_*}\ar@<-2.8ex>[l]_{\bar{j}_!}
      }
$$ is a recollement of additive categories,  where $\bar{i}_*, \bar{i}^!, \bar{i}^*, \bar{j}^*,
 \bar{j}_*$ and $\bar{j}_!$ are the restriction functors of  $i_*, i^!, i^*, j^*, j_*$ and $j_!$, respectively.
 \end{cor}
  \begin{proof}
    For every $X'\in \mathcal {X}'$, since $j^*i_*(X')=0\in \mathcal {X}''$, $i^*i_*(X')= X'\in \mathcal {X}'$,
and  $i^!i_*(X')= X'\in \mathcal {X}'$, we obtain that $i_*(\mathcal {X}')\subset \mathcal {X}$.
Thus $\mathcal {X}'=i^*i_*(\mathcal {X}')\subset i^*(\mathcal {X})$.
 On the other hand,  by definition of $\mathcal {X}$, $i^*(\mathcal {X})\subset \mathcal {X}'$,
 so $i^*(\mathcal {X})= \mathcal {X}'$.
 Similarly, we can show that $j_*(\mathcal {X}'')\subset \mathcal {X}$ and $j^*(\mathcal {X})= \mathcal {X}''$.
\par
Now we have $i_*i^*(\mathcal {X})=i_*(\mathcal {X}')\subset \mathcal {X}$,
 $j_*j^*(\mathcal {X})=j_*(\mathcal {X}'')\subset \mathcal {X}$
 and $i_*i^!(\mathcal {X})\subset i_*(\mathcal {X}')\subset \mathcal {X}$.
  It is sufficient to show that $j_!j^*(\mathcal {X})\subset \mathcal {X}$ by Theorem \ref{thm5}(1).
   It is easy to check that $j_!(\mathcal {X}'')\subset \mathcal {X}$.
   Therefore, $j_!j^*(\mathcal {X})=j_!(\mathcal {X}'')\subset \mathcal {X}$.
   This finishes the proof.
   \end{proof}

 \begin{cor}\label{cor7}
 Let
 $$ \xymatrix @C=2.7pc {
      \mathcal {A}' \ar[r]^{i_*} & \mathcal {A}   \ar@<+2.8ex>[l]_{i^!}\ar@<-2.8ex>[l]_{i^*}\ar[r]^{j^*}
      & \mathcal {A}'' \ar@<+2.8ex>[l]_{j_*}\ar@<-2.8ex>[l]_{j_!}
      }
$$
 be a recollement of additive categories, $\mathcal {X}'$  a subcategories of $\mathcal {A}'$.
Then$$ \xymatrix @C=2.7pc {
      \mathcal {A}'/\mathcal{X}' \ar[r]^{\widetilde{i}_*} & \mathcal {A}/i_*(\mathcal{X}')
      \ar@<+2.8ex>[l]_{\widetilde{i}^!}\ar@<-2.8ex>[l]_{\widetilde{i}^*}\ar[r]^-{\widetilde{j}^*}
      & \mathcal {A}''\ar@<+2.8ex>[l]_-{\widetilde{j}_*}\ar@<-2.8ex>[l]_-{\widetilde{j}_!}
      }
$$ is a recollement of additive categories.
 \end{cor}
 \begin{proof}
If we take $\mathcal {X}=i_*(\mathcal{X}')$, then $j^*\mathcal{X}=0$,
 $i_*i^*(\mathcal {X})=i_*i^!(\mathcal {X})= \mathcal {X}$ and
$j_*j^*(\mathcal {X})=j_!j^*(\mathcal {X})=0$. The result immediately follows from Theorem \ref{thm5}(2).
\end{proof}
\section{Recollement of triangulated quotient categories induced by mutation pairs}
 Let $\mathcal{D}$ be a subcategory of an additive category $\mathcal{C}$. A morphism
$f:A\rightarrow B$ in $\mathcal{C}$ is called $\mathcal{D}$-$epic$, if for any $D\in\mathcal{D}$, the sequence
$\xymatrix@C=0.5cm{\mbox{Hom}_{\mathcal{C}}(D,A) \ar[rr]^{f_\ast} && \mbox{Hom}_{\mathcal{C}}(D,B) \ar[r] & 0 }$
is exact. A $right$ $\mathcal{D}$-$approximation$ of $X$
in  $\mathcal{C}$ is a $\mathcal{D}$-epic map $f:D\rightarrow X$, with $D\in\mathcal{D}$.  The subcategory $\mathcal{D}$ is said to be a $contravariantly\ finite$  if any object $A$ of $\mathcal{C}$ has a right $\mathcal{X}$-approximation.  We can defined $\mathcal{D}$-$monic$ morphism, $left$ $\mathcal{D}$-$approximation$ and $covariantly\ finite\ subcategory$ dually.  The subcategory $\mathcal{D}$ is called {\em functorially finite} if it is both contravariantly finite and covariantly finite.

\begin{defn}(cf.\cite{[IY]},\cite{[LZ]}) \label{defn2} Let $\mathcal{C}$ be a triangulated category,  and $\mathcal{D}\subseteq \mathcal{Z}$
  be subcategories of $\mathcal{C}$. Then the pair $(\mathcal{Z}, \mathcal{Z})$ is called a $\mathcal{D}$-$mutation\ pair$ if the following conditions are satisfied.

(1) For any object $X\in \mathcal{Z}$,  there exists a triangle
  $$\xymatrix@C=0.5cm{
  X \ar[r]^{f} & D \ar[r]^{g} & Y \ar[r]^{h} & TX  }$$
  where  $Y\in \mathcal{Z}$, $D\in\mathcal{D}$, $f$ is a left $\mathcal{D}$-approximation and $g$ is a
  right $\mathcal{D}$-approximation.
 \par
  (2) For any object $Y\in \mathcal{Z}$,  there exists a triangle
  $$\xymatrix@C=0.5cm{
  X \ar[r]^{f} & D \ar[r]^{g} & Y \ar[r]^{h} & TX  }$$
  where  $X\in \mathcal{Z}$, $D\in\mathcal{D}$, $f$ is a left $\mathcal{D}$-approximation and $g$ is a
  right $\mathcal{D}$-approximation.
\end{defn}

 \begin{lem}\label{lem3.4}
 Let $\mathcal{C}$, $\mathcal{C}'$ be triangulated categories, and $F:\mathcal{C}\rightarrow \mathcal{C}'$  a full and exact functor.  If $\mathcal{D}\subset \mathcal{Z}$  are subcategories of $\mathcal{C}$, and  $(\mathcal{Z}, \mathcal{Z})$ is a $\mathcal{D}$-mutation  pair,
 then  $(F\mathcal{Z}, F\mathcal{Z})$ is a $F\mathcal{D}$-mutation  pair.
\end{lem}
\begin{proof}
 For any object $FX\in F\mathcal{Z}$, where $X\in\mathcal{\mathcal{Z}}$,  there exists a triangle in $\mathcal{C}$
  $$\xymatrix@C=0.5cm{
  X \ar[r]^{f} & D \ar[r]^{g} & Y \ar[r]^{h} & TX  }$$
  where  $Y\in \mathcal{Z}$, $D\in\mathcal{D}$, $f$ is a left $\mathcal{D}$-approximation and $g$ is a
  right $\mathcal{D}$-approximation.
 Since $F$ is exact,  we have a triangle in $\mathcal{C}'$
  $$\xymatrix@C=0.5cm{
  FX \ar[r]^{Ff} & FD \ar[r]^{Fg} & FY \ar[r]^{ } & TFX.}$$
  Let $u'\in \mbox{Hom}_{\mathcal{C}'}(FX, FD')$. Since $F$ is full, there is a morphism
  $u\in \mbox{Hom}_{\mathcal{C}}(X, D')$ such that $Fu=u'$. Since $f$ is  a left $\mathcal{D}$-approximation,
  there is a morphism $d\in \mbox{Hom}_{\mathcal{C}}(D, D')$, such that $u=df$.
  Hence $u'=Fu=Fd Ff$. Then $Ff$ is a left $F\mathcal{D}$-approximation.
  Similarly, $Fg$ is a right $F\mathcal{D}$-approximation.
  This shows that $(F\mathcal{Z}, F\mathcal{Z})$ satisfies condition (1) of Definition \ref{defn2}.
  Similarly, we can show that it also satisfies condition (2) of Definition \ref{defn2}.
  Therefore $(F\mathcal{Z}, F\mathcal{Z})$ is a $F\mathcal{D}$-mutation  pair. 
This finishes the proof.
\end{proof}

\par
Let $(\mathcal{Z}, \mathcal{Z})$ be a $\mathcal{D}$-mutation pair. For any object $X\in\mathcal{Z}$, fix a triangle
  $$\xymatrix@C=0.5cm{
  X \ar[r]^{\alpha_X} & D_X \ar[r]^{\beta_X} & M \ar[r]^{\gamma_X} & TX  }$$
  where  $M\in \mathcal{Z}$, $D_X\in\mathcal{D}$, $\alpha_X$ is a left $\mathcal{D}$-approximation and $\beta_X$ is a
  right $\mathcal{D}$-approximation. We can define an equivalent functor $\sigma: \mathcal{Z}/\mathcal{D}\rightarrow\mathcal{Z}/\mathcal{D}$ such that $\sigma X=M$.
Let $\xymatrix@C=0.5cm{X \ar[r]^{f} & Y \ar[r]^{g} & Z \ar[r]^{h} & TX  }$ be a triangle in $\mathcal{C}$ with $X,Y,Z\in\mathcal{Z}$, where $f$ is $\mathcal{D}$-monic.
Then there exists a commutative diagram where rows are triangles.
$$\xymatrix {
X\ar@{=}[d] \ar[r]^{f} & Y\ar[d]^y \ar[r]^{g} & Z\ar[d]^z \ar[r]^{h} & TX\ar@{=}[d]\\
X \ar[r]^{\alpha_X} & D_X \ar[r]^{\beta_X} & \sigma X \ar[r]^{\gamma_X }& TX
 }$$
We have the following sextuple in the quotient category $\mathcal{Z}/\mathcal{D}$
$$(*) \hskip 10pt \xymatrix@C=0.5cm{X \ar[r]^{\overline{f}} & Y \ar[r]^{\overline{g}} & Z \ar[r]^{\overline{z}} & \sigma X  }$$
We define $\Phi$ the class of triangles in $\mathcal{Z}/\mathcal{D}$ as the sextuples which are isomorphic to $(*)$.
With notation as above, we have the following lemma.

\begin{lem}(\cite{[IY]})\label{lem5}
Let $(\mathcal{Z}, \mathcal{Z})$ be a $\mathcal{D}$-mutation pair in a triangulated category $\mathcal{C}$. If $\mathcal{Z}$ is extension-closed, then $(\mathcal{Z}/\mathcal{D},\sigma, \Phi)$ is a triangulated category.
\end{lem}

Now assume that $(\mathcal{C}, \mathcal{C})$ is a $\mathcal{D}$-mutation pair in a triangulated category $\mathcal{C}$. Then $\mathcal{C}/\mathcal{D}$ is a triangulated category by Lemma \ref{lem5}.
Let  $F:\mathcal{C}\rightarrow \mathcal{C}'$ be an exact functor between triangulated categories.
  If $F$ is full and dense, then $(\mathcal{C}', \mathcal{C}')$ is a $F\mathcal{D}$-mutation pair by Lemma \ref{lem3.4}, thus $\mathcal{C}'/F{\mathcal {D}}$ is also a triangulated category by Lemma \ref{lem5}.
It is easy to see that $F$ induces an additive functor $\widetilde{F}:\mathcal{C}/\mathcal{D} \rightarrow \mathcal{C}'/F\mathcal{D}$
 satisfying the following  commutative  diagram
 $$\xymatrix {
 \mathcal{C} \ar[d]_{Q_{\mathcal{C}}} \ar[rr]^{F} && \mathcal{C'} \ar[d]^{Q_{\mathcal{C}'}} \\
  \mathcal{C}/{\mathcal {D}} \ar[rr]^{\widetilde{F}} && \mathcal{C}'/F{\mathcal {D}}.  }
$$
Moreover, we have the following lemma.
\begin{lem}\label{lem3.6} The functor
$\widetilde{F}:\mathcal{C}/\mathcal{D} \rightarrow \mathcal{C}'/F\mathcal{D}$ is exact.
\end{lem}
\begin{proof}
 Let
 $\xymatrix@C=0.5cm{X \ar[r]^{\overline{f}} & Y \ar[r]^{\overline{g}} & Z \ar[r]^{\overline{z}} & \sigma X  }$
be a standard triangle in $\mathcal{C}/\mathcal{D}$, then there exists the following  commutative  diagram of triangles in $\mathcal{C}$.
$$\xymatrix {
X\ar@{=}[d] \ar[r]^{f} & Y\ar[d]^y \ar[r]^{g} & Z\ar[d]^z \ar[r]^{h} & TX\ar@{=}[d]\\
X \ar[r]^{\alpha_X} & D_X \ar[r]^{\beta_X} & \sigma X \ar[r]^{\gamma_X }& TX
 }$$
Thus we have the following commutative  diagram of triangles in $\mathcal{C'}$
$$\xymatrix {
FX\ar@{=}[d] \ar[r]^{Ff} & FY\ar[d]^{Fy} \ar[r]^{Fg} & FZ\ar[d]^{Fz} \ar[r]^{\phi_X\cdot Fh} & TFX\ar@{=}[d]\\
FX \ar[r]^{F(\alpha_X)} & FD_X \ar[r]^{F(\beta_X)} & F\sigma X \ar[r]^{\phi_X\cdot F(\gamma_X) }& TFX
 }$$
where $\phi_X:FTX\rightarrow TFX$ is a natural equivalence. Note that $F(\alpha_X)$ is a left $F\mathcal{D}$-approximation
and  $F(\beta_X)$ is a right $F\mathcal{D}$-approximation, thus $\sigma FX=F\sigma X$.
By definition we obtain that $\xymatrix@C=0.5cm{FX \ar[r]^{\overline{Ff}} & FY \ar[r]^{\overline{Fg}} & FZ \ar[r]^{\overline{Fz}} & \sigma(FX)}$ is a triangle in $\mathcal{C}'/F\mathcal{D}$. Namely,
$\xymatrix@C=0.5cm{\widetilde{F}X \ar[r]^{\widetilde{F}\overline{f}} & \widetilde{F}Y \ar[r]^{\widetilde{F}\overline{g}} &
\widetilde{F}Z\ar[r]^{\widetilde{F}\overline{z}\:}&\sigma(\widetilde{F}X)}$
is a triangle in $\mathcal{C}'/F\mathcal{D}$.
Therefore $\widetilde{F}:\mathcal{C}/\mathcal{D} \rightarrow \mathcal{C}'/F\mathcal{D}$ is an exact functor.
\end{proof}
\begin{thm}Let
 $$ \xymatrix @C=2.7pc {
      \mathcal {C}' \ar[r]^{i_*} & \mathcal {C}   \ar@<+2.8ex>[l]_{i^!}\ar@<-2.8ex>[l]_{i^*}\ar[r]^{j^*}
      & \mathcal {C}'' \ar@<+2.8ex>[l]_{j_*}\ar@<-2.8ex>[l]_{j_!}
      }
$$
 be a recollement of triangulated categories,  $\mathcal{D}\subset$ Ker$j^*$
  be subcategories of $\mathcal{C}$  and
  $(\mathcal{C}, \mathcal{C})$ be a $\mathcal{D}$-mutation  pair. Then
\par
$$ \xymatrix @C=2.7pc {
      \mathcal {C}'/i^*(\mathcal{D}) \ar[r]^{\widetilde{i}_*} & \mathcal {C}/\mathcal{D}
      \ar@<+2.8ex>[l]_{\widetilde{i}^!}\ar@<-2.8ex>[l]_{\widetilde{i}^*}\ar[r]^{\widetilde{j}^*}
      & \mathcal {C}''\ar@<+2.8ex>[l]_{\widetilde{j}_*}\ar@<-2.8ex>[l]_{\widetilde{j}_!}
      }
$$ is a recollement of triangulated categories.
\end{thm}
\begin{proof}
Since $\mathcal{D}\subset$ Ker$j^*=$Im$i_*$, there exists $\mathcal{D}'$ of subcategory of $\mathcal {C}'$
such that  $i_*(\mathcal{D}')=\mathcal{D}$. Thus
$i_*i^*(\mathcal {D})=i_*i^*i_*(\mathcal{D}')=i_*(\mathcal{D}')= \mathcal {D}$.
By Corollary \ref{cor7}, the following diagram  $$ \xymatrix @C=2.7pc {
      \mathcal {C}'/i^*(\mathcal{D}) \ar[r]^{\widetilde{i}_*} & \mathcal {C}/\mathcal{D}
      \ar@<+2.8ex>[l]_{\widetilde{i}^!}\ar@<-2.8ex>[l]_{\widetilde{i}^*}\ar[r]^{\widetilde{j}^*}
      & \mathcal {C}''\ar@<+2.8ex>[l]_{\widetilde{j}_*}\ar@<-2.8ex>[l]_{\widetilde{j}_!}
      }
$$ is a recollement of additive categories.

Since $(\mathcal{C}, \mathcal{C})$ is a $\mathcal{D}$-mutation  pair, $\mathcal{C}/\mathcal{D}$ is a triangulated category.
Since the functors $i^*$ and $j^*$ are full and dense, $\mathcal {C}'/i^*(\mathcal{D})$ is a triangulated category by Lemma \ref{lem3.4} and $\widetilde{j^*}, \widetilde{i^*}$ are exact functors by  Lemma \ref{lem3.6}.
Since the left and right adjoint functors of an exact functor are exact functors too,
 $ \widetilde{i}_*, \widetilde{i}^!, \widetilde{j}_!, \widetilde{j}^*$ are exact functors.
 This finishes the proof by Remark \ref{rem}(1).
 \end{proof}
 \par
 From now on, assume that $k$ is an algebraically closed field, and all categories are  $k$-linear categories with finite dimensional Hom spaces and split idempotents. Let $\mathcal{C}$ be a triangulated category with a Serre functor $S$. Then $\mathcal{C}$ has AR-triangles and we denote by $\tau$ the Auslander-Reiten translation functor. If $\mathcal{X}$ is a functorially finite subcategory of $\mathcal{C}$ such that $\tau\mathcal{X}=\mathcal{X}$, then $(\mathcal{C},\mathcal{C})$ is a $\mathcal{X}$-mutation pair by \cite[Lemma 2.2]{[J]}.
 Thus $\mathcal{C}/\mathcal{X}$ is a triangulated category.
According to Theorem \ref{thm5}, we have the following corollary.
 \begin{cor}
 Let
 $$ \xymatrix @C=2.7pc {
      \mathcal {C}' \ar[r]^{i_*} & \mathcal {C}   \ar@<+2.8ex>[l]_{i^!}\ar@<-2.8ex>[l]_{i^*}\ar[r]^{j^*}
      & \mathcal {C}'' \ar@<+2.8ex>[l]_{j_*}\ar@<-2.8ex>[l]_{j_!}
      }
$$
 be a recollement of triangulated categories,
  $\mathcal{X}$ a  functorially finite subcategory of $\mathcal{C}$ such that $\tau\mathcal{X}=\mathcal{X}$ and $\mathcal{X}\subset$ Ker$j^*$,
 then
\par
$$ \xymatrix @C=2.7pc {
      \mathcal {C}'/i^*(\mathcal{X}) \ar[r]^{\widetilde{i}_*} & \mathcal {C}/\mathcal{X}
      \ar@<+2.8ex>[l]_{\widetilde{i}^!}\ar@<-2.8ex>[l]_{\widetilde{i}^*}\ar[r]^{\widetilde{j}^*}
      & \mathcal {C}''\ar@<+2.8ex>[l]_{\widetilde{j}_*}\ar@<-2.8ex>[l]_{\widetilde{j}_!}
      }
$$ is a recollement of triangulated categories.
\end{cor}

\end{document}